\documentclass{amsart}
\usepackage{amsmath,amssymb,amsthm,bm,bbm}
\usepackage{graphicx,enumerate}
\usepackage{layout}
\usepackage{longtable}

\theoremstyle{plain}
\newtheorem{theorem}{Theorem}[section]
\newtheorem*{theorem-nn}{Theorem}
\newtheorem{lemma}[theorem]{Lemma}

\newtheorem*{proposition-nn}{Proposition}

\theoremstyle{definition}
\newtheorem{definition}[theorem]{Definition}
\newtheorem{example}[theorem]{Example}
\newtheorem{remark}[theorem]{Remark}
\newtheorem*{acknowledgments}{Acknowledgments}

\theoremstyle{remark}

\newcommand{\bZ}{\mathbbm{Z}}\newcommand{\bQ}{\mathbbm{Q}}
\newcommand{\bG}{\mathbbm{G}}
\newcommand{\bF}{\mathbbm{F}}

\newcommand{\cC}{\mathcal{C}}\newcommand{\cD}{\mathcal{D}}
\newcommand{\cH}{\mathcal{H}}\newcommand{\cS}{\mathcal{S}}
\newcommand{\GL}{{\rm GL}}\newcommand{\SL}{{\rm SL}}
\newcommand{\PGL}{{\rm PGL}}\newcommand{\PSL}{{\rm PSL}}
\newcounter{sub}
{\begin{list}{(\arabic{sub})}{\usecounter{sub}%
\setlength{\leftmargin}{2em}}}{\end{list}}

\def\rank{\mbox{rank }}

\title[Rationality problem for norm one tori]
{Rationality problem for norm one tori}

\author[A. Hoshi]{Akinari Hoshi}
\address{Department of Mathematics, Niigata University, Niigata 950-2181, Japan}
\email{hoshi@math.sc.niigata-u.ac.jp}

\author[A. Yamasaki]{Aiichi Yamasaki}
\address{Department of Mathematics, Kyoto University, Kyoto 606-8502, Japan}
\email{aiichi.yamasaki@gmail.com}

\thanks{
This work was partially supported by JSPS KAKENHI Grant Numbers 
25400027, 16K05059, 19K03418.
}


\begin{document}
\begin{abstract}
We classify stably/retract rational norm one tori in dimension $p-1$ where $p$ is a prime number and in dimension up to ten with some minor exceptions. 
\end{abstract}

\maketitle

\tableofcontents

\section{Introduction}\label{seInt}

Let $k$ be a field and $K$ 
be a finitely generated field extension of $k$. 
A field $K$ is called {\it rational over $k$} 
(or {\it $k$-rational} for short) 
if $K$ is purely transcendental over $k$, 
i.e. $K$ is isomorphic to $k(x_1,\ldots,x_n)$, 
the rational function field over $k$ with $n$ variables $x_1,\ldots,x_n$ 
for some integer $n$. 
$K$ is called {\it stably $k$-rational} 
if $K(y_1,\ldots,y_m)$ is $k$-rational for some algebraically 
independent elements $y_1,\ldots,y_m$ over $K$. 
Two fields 
$K$ and $K^\prime$ are called {\it stably $k$-isomorphic} if 
$K(y_1,\ldots,y_m)\simeq K^\prime(z_1,\ldots,z_n)$ over $k$ 
for some algebraically independent elements $y_1,\ldots,y_m$ over $K$ 
and $z_1,\ldots,z_n$ over $K^\prime$. 
When $k$ is an infinite field, 
$K$ is called {\it retract $k$-rational} 
if there is a $k$-algebra $R$ contained in $K$ such that 
(i) $K$ is the quotient field of $R$, and (ii) 
the identity map $1_R : R\rightarrow R$ factors through a localized 
polynomial ring over $k$, i.e. there is an element $f\in k[x_1,\ldots,x_n]$, 
which is the polynomial ring over $k$, and there are $k$-algebra 
homomorphisms $\varphi : R\rightarrow k[x_1,\ldots,x_n][1/f]$ 
and $\psi : k[x_1,\ldots,x_n][1/f]\rightarrow R$ satisfying 
$\psi\circ\varphi=1_R$ (cf. \cite{Sal84}). 
$K$ is called {\it $k$-unirational} 
if $k\subset K\subset k(x_1,\ldots,x_n)$ for some integer $n$. 
It is not difficult to see that 
``$k$-rational'' $\Rightarrow$ ``stably $k$-rational'' $\Rightarrow$ 
``retract $k$-rational'' $\Rightarrow$ ``$k$-unirational''. 

Let $L$ be a finite Galois extension of $k$ and $G={\rm Gal}(L/k)$ 
be the Galois group of the extension $L/k$. 
Let $M=\bigoplus_{1\leq i\leq n}\bZ\cdot u_i$ be a $G$-lattice with 
a $\bZ$-basis $\{u_1,\ldots,u_n\}$, 
i.e. finitely generated $\bZ[G]$-module 
which is $\bZ$-free as an abelian group. 
Let $G$ act on the rational function field $L(x_1,\ldots,x_n)$ 
over $L$ with $n$ variables $x_1,\ldots,x_n$ by 
\begin{align}
\sigma(x_i)=\prod_{j=1}^n x_j^{a_{i,j}},\quad 1\leq i\leq n\label{acts}
\end{align}
for any $\sigma\in G$, when $\sigma (u_i)=\sum_{j=1}^n a_{i,j} u_j$, 
$a_{i,j}\in\bZ$. 
The field $L(x_1,\ldots,x_n)$ with this action of $G$ will be denoted 
by $L(M)$.
There is the duality between the category of $G$-lattices 
and the category of algebraic $k$-tori which split over $L$ 
(see \cite[Section 1.2]{Ono61}, \cite[page 27, Example 6]{Vos98}). 
In fact, if $T$ is an algebraic $k$-torus, then the character 
group $X(T)={\rm Hom}(T,\bG_m)$ of $T$ may be regarded as a $G$-lattice. 
Conversely, for a given $G$-lattice $M$, there exists an algebraic $k$-torus 
$T$ which splits over $L$ such that $X(T)$ is isomorphic to $M$ as a $G$-lattice. 

The invariant field $L(M)^G$ of $L(M)$ under the action of $G$ 
may be identified with the function field of the algebraic $k$-torus $T$. 
Note that the field $L(M)^G$ is always $k$-unirational 
(see \cite[page 40, Example 21]{Vos98}).
Tori of dimension $n$ over $k$ correspond bijectively 
to the elements of the set $H^1(\mathcal{G},\GL_n(\bZ))$ 
where $\mathcal{G}={\rm Gal}(k_{\rm s}/k)$ since 
${\rm Aut}(\bG_m^n)=\GL_n(\bZ)$. 
The $k$-torus $T$ of dimension $n$ is determined uniquely by the integral 
representation $h : \mathcal{G}\rightarrow \GL_n(\bZ)$ up to conjugacy, 
and the group $h(\mathcal{G})$ is a finite subgroup of $\GL_n(\bZ)$ 
(see \cite[page 57, Section 4.9]{Vos98})). 

Let $K/k$ be a separable field extension of degree $n$ 
and $L/k$ be the Galois closure of $K/k$. 
Let $G={\rm Gal}(L/k)$ and $H={\rm Gal}(L/K)$ with $[G:H]=n$. 
The Galois group $G$ may be regarded as a transitive subgroup of 
the symmetric group $S_n$ of degree $n$. 
We may assume that 
$H$ is the stabilizer of one of the letters in $G$, 
i.e. $L=k(\theta_1,\ldots,\theta_n)$ and $K=L^H=k(\theta_i)$ 
where $1\leq i\leq n$. 

Let $R^{(1)}_{K/k}(\bG_m)$ be the norm one torus of $K/k$,
i.e. the kernel of the norm map $R_{K/k}(\bG_m)\rightarrow \bG_m$ where 
$R_{K/k}$ is the Weil restriction (see \cite[page 37, Section 3.12]{Vos98}). 
The norm one torus $R^{(1)}_{K/k}(\bG_m)$ has the 
Chevalley module $J_{G/H}$ as its character module 
and the field $L(J_{G/H})^G$ as its function field 
where $J_{G/H}=(I_{G/H})^\circ={\rm Hom}_\bZ(I_{G/H},\bZ)$ 
is the dual lattice of $I_{G/H}={\rm Ker}\ \varepsilon$ and 
$\varepsilon : \bZ[G/H]\rightarrow \bZ$ is the augmentation map 
(see \cite[Section 4.8]{Vos98}). 
We have the exact sequence $0\rightarrow \bZ\rightarrow \bZ[G/H]
\rightarrow J_{G/H}\rightarrow 0$ and rank $J_{G/H}=n-1$. 
Write $J_{G/H}=\oplus_{1\leq i\leq n-1}\bZ x_i$. 
Then the action of $G$ on $L(J_{G/H})=L(x_1,\ldots,x_{n-1})$ is 
of the form 
(\ref{acts}). 

Let $T=R^{(1)}_{K/k}(\bG_m)$ be the norm one torus defined by $K/k$. 
Let $S_n$ (resp. $A_n$, $D_n$, $C_n$) be the symmetric 
(resp. the alternating, the dihedral, the cyclic) group 
of degree $n$ of order $n!$ (resp. $n!/2$, $2n$, $n$).

The rationality problem for norm one tori is investigated 
by \cite{EM75}, \cite{CTS77}, \cite{Hur84}, \cite{CTS87}, 
\cite{LeB95}, \cite{CK00}, \cite{LL00}, \cite{Flo}, \cite{End11} 
and \cite{HY17}. 
\begin{theorem}[{Endo and Miyata \cite[Theorem 1.5]{EM75}, Saltman \cite[Theorem 3.14]{Sal84}}]\label{th13-1}
Let $K/k$ be a finite Galois field extension and $G={\rm Gal}(K/k)$. 
Then the following conditions are equivalent:\\
{\rm (i)} $R^{(1)}_{K/k}(\bG_m)$ is retract $k$-rational;\\
{\rm (ii)} all the Sylow subgroups of $G$ are cyclic. 
\end{theorem}

\begin{theorem}[{Endo and Miyata \cite[Theorem 2.3]{EM75}, Colliot-Th\'{e}l\`{e}ne and Sansuc \cite[Proposition 3]{CTS77}}]\label{th13-2}
Let $K/k$ be a finite Galois field extension and $G={\rm Gal}(K/k)$. 
Then the following conditions are equivalent:\\
{\rm (i)} $R^{(1)}_{K/k}(\bG_m)$ is stably $k$-rational;\\
{\rm (ii)} all the Sylow subgroups of $G$ are cyclic and $H^4(G,\bZ)\simeq \widehat H^0(G,\bZ)$ 
where $\widehat H$ is the Tate cohomology;\\
{\rm (iii)} $G=C_m$ or $G=C_n\times \langle\sigma,\tau\mid\sigma^k=\tau^{2^d}=1,
\tau\sigma\tau^{-1}=\sigma^{-1}\rangle$ where $d\geq 1, k\geq 3$, 
$n,k$: odd, and ${\rm gcd}\{n,k\}=1$;\\
{\rm (iv)} $G=\langle s,t\mid s^m=t^{2^d}=1, tst^{-1}=s^r, m: odd,\ 
r^2\equiv 1\pmod{m}\rangle$.
\end{theorem}
\begin{theorem}[Endo {\cite[Theorem 2.1]{End11}}]\label{th14}
Let $K/k$ be a finite non-Galois, separable field extension 
and $L/k$ be the Galois closure of $K/k$. 
Assume that the Galois group of $L/k$ is nilpotent. 
Then the norm one torus $R^{(1)}_{K/k}(\bG_m)$ is not 
retract $k$-rational.
\end{theorem}
\begin{theorem}[Endo {\cite[Theorem 3.1]{End11}}]\label{th15}
Let $K/k$ be a finite non-Galois, separable field extension 
and $L/k$ be the Galois closure of $K/k$. 
Let $G={\rm Gal}(L/k)$ and $H={\rm Gal}(L/K)\leq G$. 
Assume that all the Sylow subgroups of $G$ are cyclic. 
Then the norm one torus $R^{(1)}_{K/k}(\bG_m)$ is retract $k$-rational, 
and the following conditions are equivalent:\\
{\rm (i)} 
$R^{(1)}_{K/k}(\bG_m)$ is stably $k$-rational;\\
{\rm (ii)} 
$G=D_n$ with $n$ odd $(n\geq 3)$ 
or $G=C_m\times D_n$ where $m,n$ are odd, 
$m,n\geq 3$, ${\rm gcd}\{m,n\}=1$, and $H\leq D_n$ is of order $2$;\\
{\rm (iii)} 
$H=C_2$ and $G\simeq C_r\rtimes H$, $r\geq 3$ odd, where 
$H$ acts non-trivially on $C_r$. 
\end{theorem}
\begin{theorem}[{Colliot-Th\'{e}l\`{e}ne and Sansuc \cite[Proposition 9.1]{CTS87}, 
\cite[Theorem 3.1]{LeB95}, 
\cite[Proposition 0.2]{CK00}, \cite{LL00}, 
Endo \cite[Theorem 4.1]{End11}, see also 
\cite[Remark 4.2 and Theorem 4.3]{End11}}]\label{thS}
Let $K/k$ be a non-Galois separable field extension 
of degree $n$ and $L/k$ be the Galois closure of $K/k$. 
Assume that ${\rm Gal}(L/k)=S_n$, $n\geq 3$, 
and ${\rm Gal}(L/K)=S_{n-1}$ is the stabilizer of one of the letters 
in $S_n$.\\
{\rm (i)}\ 
$R^{(1)}_{K/k}(\bG_m)$ is retract $k$-rational 
if and only if $n$ is a prime number;\\
{\rm (ii)}\ 
$R^{(1)}_{K/k}(\bG_m)$ is $($stably$)$ $k$-rational 
if and only if $n=3$.
\end{theorem}
%
%
\begin{theorem}[Endo {\cite[Theorem 4.4]{End11}, Hoshi and Yamasaki \cite[Corollary 1.11]{HY17}}]\label{thA}
Let $K/k$ be a non-Galois separable field extension 
of degree $n$ and $L/k$ be the Galois closure of $K/k$. 
Assume that ${\rm Gal}(L/k)=A_n$, $n\geq 4$, 
and ${\rm Gal}(L/K)=A_{n-1}$ is the stabilizer of one of the letters 
in $A_n$.\\
{\rm (i)}\ 
$R^{(1)}_{K/k}(\bG_m)$ is retract $k$-rational 
if and only if $n$ is a prime number.\\
{\rm (ii)}\ $R^{(1)}_{K/k}(\bG_m)$ is stably $k$-rational 
if and only if $n=5$.
\end{theorem}
Let $nTm$ be the $m$-th transitive subgroup of $S_n$. 
There exist $2$ (resp. $5$, $5$, $16$, $7$, $50$, $34$, $45$, $8$) 
transitive subgroups of $S_3$ (resp. $S_4$, $S_5$, $S_6$, $S_7$, 
$S_8$, $S_9$, $S_{10}$, $S_{11}$) 
(see Butler and McKay \cite{BM83}, \cite{GAP}). 
Let $F_{pm}\simeq C_p\rtimes C_m\leq S_p$ 
be the Frobenius group of order $pm$ where $m\mid p-1$.


%
\begin{theorem}[{Hoshi and Yamasaki \cite[Theorem 1.10, Theorem 1.14, Theorem 8.5]{HY17}}]\label{th5to11}
Let $K/k$ be a separable field extension of degree $n$ 
and $L/k$ be the Galois closure of $K/k$. 
Let $G={\rm Gal}(L/k)$ be a transitive subgroup of $S_n$ 
and $H={\rm Gal}(L/K)$ with $[G:H]=n$. 
Then 
a classification of stably/retract rational 
norm one tori $T=R_{K/k}^{(1)}(\bG_m)$ in dimension $n-1$ for $n=5,6,7,11$ 
is given as follows:\\ 
{\rm (1)} The case $5Tm$ $(1\leq m\leq 5)$.\\
{\rm (i)} $T$ is stably $k$-rational 
for $5T1\simeq C_5$, $5T2\simeq D_5$ and $5T4\simeq A_5$;\\
{\rm (ii)} $T$ is not stably but retract $k$-rational 
for $5T3\simeq F_{20}$ and $5T5\simeq S_5$.\\
{\rm (2)} The case $6Tm$ $(1\leq m\leq 16)$.\\
{\rm (i)} $T$ is stably $k$-rational 
for $6T1\simeq C_6$, $6T2\simeq S_3$ and $6T3\simeq D_6$;\\
{\rm (ii)} $T$ is not retract $k$-rational 
for $6Tm$ with $4\leq m\leq 16$ which is isomorphic to 
$A_4$, 
$C_3\times S_3$, 
$C_2\times A_4$, 
$S_4$, 
$S_4$, 
$S_3^2$, 
$C_3^2\rtimes C_4$, 
$C_2\times S_4$, 
$A_5$, 
$S_3^2\rtimes C_2$
$S_5$, 
$A_6$, 
$S_6$ respectively.\\
{\rm (3)} The case $7Tm$ $(1\leq m\leq 7)$.\\
{\rm (i)} $T$ is stably $k$-rational 
for $7T1\simeq C_7$ and $7T2\simeq D_7$;\\
{\rm (ii)} $T$ is not stably but retract $k$-rational 
for $7T3\simeq F_{21}$, $7T4\simeq F_{42}$, $7T5\simeq \PSL_3(\bF_2)\simeq\PSL_2(\bF_{7})$, 
$7T6\simeq A_7$ and $7T7\simeq S_7$.\\
{\rm (4)} The case $11Tm$ $(1\leq m\leq 8)$.\\ 
{\rm (i)} $T$ is stably $k$-rational 
for $11T1\simeq C_{11}$ and $11T2\simeq D_{11}$;\\
{\rm (ii)} $T$ is not stably but retract $k$-rational 
for $11T3\simeq F_{55}$, $11T4\simeq F_{110}$, $11T5\simeq \PSL_2(\bF_{11})$, 
$11T6\simeq M_{11}$, $11T7\simeq A_{11}$ and $11T8\simeq S_{11}$ where 
$M_{11}$ is the Mathieu group of degree $11$.
\end{theorem}
%

%


\begin{theorem}[see {Dixon and Mortimer \cite[page 99]{DM96}}]\label{thTp}
Let $p$ be a prime number and $G\leq S_p$ be a transitive subgroup.\\ 
{\rm (1)} If $G$ is solvable, then $G\simeq C_p\rtimes C_m\leq S_p$ 
is the Frobenius group of order $pm$ with $m\mid p-1$.\\ 
{\rm (2)} If $G$ is not solvable, then $G$ is one of the following:\\
{\rm (i)} $G=S_p$ or $G=A_p\leq S_{p}$;\\
{\rm (ii)} $G=\PSL_2(\bF_{11})\leq S_{11}$;\\
{\rm (iii)} $G=M_{11}\leq S_{11}$ or $G=M_{23}\leq S_{23}$ where 
$M_p$ is the Mathieu group of degree $p$;\\
{\rm (iv)} $\PSL_d(\bF_q)\leq G\leq 
{\rm P\Gamma L}_d(\bF_q)\simeq \PGL_d(\bF_q)\rtimes C_e$ 
where $p=\frac{q^d-1}{q-1}$, $q=l^e$ is a prime power and 
\end{theorem}

Theorem \ref{thmain1} and Theorem \ref{thmain2} are the main results 
of this paper. 

\begin{theorem}\label{thmain1}
Let $p\geq 3$ be a prime number, 
$K/k$ be a separable field extension of degree $p$ 
and $L/k$ be the Galois closure of $K/k$. 
Let $G={\rm Gal}(L/k)$ be a transitive subgroup of $S_p$ 
and $H={\rm Gal}(L/K)$ with $[G:H]=p$.
Then norm one tori $T=R_{K/k}^{(1)}(\bG_m)$ of dimension $p-1$ 
are retract $k$-rational and 
a stably rational classification of $T$ is given as follows:\\ 
{\rm (1)} $T$ is stably $k$-rational 
for $G\simeq C_p\leq S_p$ and $G\simeq D_p\leq S_p$;\\
{\rm (2)} $T$ is not stably $k$-rational 
for $G\simeq C_p\rtimes C_m\leq S_p$ with $3\leq m\mid p-1$;\\
{\rm (3)} 
$T$ is not stably $k$-rational for $G\simeq S_p$ where $p\geq 5$;\\
{\rm (4)} $T$ is stably $k$-rational for $G\simeq A_5\leq S_5$ 
and $T$ is not stably $k$-rational for $G\simeq A_p\leq S_p$ where $p\geq 7$;\\
{\rm (5)} $T$ is not stably $k$-rational 
for $G\simeq \PSL_2(\bF_{11})\leq S_{11}$;\\
{\rm (6)} $T$ is not stably $k$-rational 
for $G\simeq M_{11}\leq S_{11}$ and $G\simeq M_{23}\leq S_{23}$;\\
{\rm (7)} 
$T$ is not stably $k$-rational 
for $\PSL_d(\bF_q)\leq G\leq 
{\rm P\Gamma L}_d(\bF_q)\simeq \PGL_d(\bF_q)\rtimes C_e$
where $d\geq  3$, $p=\frac{q^d-1}{q-1}$ and $q=l^e$ is a prime power;\\
{\rm (8)} 
$T$ is not stably $k$-rational 
for $\PSL_2(\bF_{2^e})< G\leq 
{\rm P\Gamma L}_2(\bF_{2^e})\simeq \PSL_2(\bF_{2^e})\rtimes C_e$ 
where $p=2^e+1$ is a Fermat prime. 
\end{theorem}
\begin{remark}
We do not know whether $T$ is stably $k$-rational 
in the case {\rm (8)} in Theorem \ref{thmain1} when 
$G=\PSL_2(\bF_{2^e})$ and $p\geq 17$. 
Note that for Fermat primes $p=3$ and $5$, 
$T$ is stably $k$-rational for $G=\PSL_2(\bF_{2^e})$ 
by Theorem \ref{thmain1} (1), (4) (note that 
$\PSL_2(\bF_2)\simeq D_3\simeq S_3$, 
$\PSL_2(\bF_4)=\PGL_2(\bF_4)\simeq A_5$).
\end{remark}
\begin{theorem}\label{thmain2}
Let $K/k$ be a separable field extension of degree $n$ 
and $L/k$ be the Galois closure of $K/k$. 
Let $G={\rm Gal}(L/k)$ be a transitive subgroup of $S_n$ 
and $H={\rm Gal}(L/K)$ with $[G:H]=n$. 
Then 
a classification of stably/retract rational 
norm one tori $T=R_{K/k}^{(1)}(\bG_m)$ in dimension $n-1$ 
for $n=8,9,10$ is given as follows:\\ 
{\rm (1)} The case $8Tm$ $(1\leq m\leq 50)$.\\
{\rm (i)} $T$ is stably $k$-rational for $8T1\simeq C_8$;\\
{\rm (ii)} $T$ is not retract $k$-rational for $8Tm$ with $2\leq m\leq 50$.\\
{\rm (2)} The case $9Tm$ $(1\leq m\leq 34)$.\\
{\rm (i)} $T$ is stably $k$-rational 
for $9T1\simeq C_9$ and $9T3\simeq D_9$;\\
{\rm (ii)} $T$ is retract $k$-rational for $9T27\simeq \PSL_2(\bF_8)$;\\
{\rm (iii)} $T$ is not retract $k$-rational for $9Tm$ with 
$2\leq m\leq 34\ \textrm{and}\ m\neq 3, 27$.\\
{\rm (3)} The case $10Tm$ $(1\leq m\leq 45)$.\\
{\rm (i)} $T$ is stably $k$-rational for 
$10T1\simeq C_{10}$, $10T2\simeq D_5$ and $10T3\simeq D_{10}$;\\
{\rm (ii)} $T$ is retract $k$-rational for $10T11\simeq A_5\times C_2$;\\ 
{\rm (iii)} $T$ is not stably but retract $k$-rational for 
$10T4\simeq F_{20}$, $10T5\simeq F_{20}\times C_2$, 
$10T12\simeq S_5$ and $10T22\simeq S_5\times C_2$;\\
{\rm (iv)} $T$ is not retract $k$-rational for $10Tm$ with 
$6\leq m\leq 45\ \textrm{and}\ m\neq 11,12,22$.
\end{theorem}

\begin{remark}
(1) In the cases (2)-(ii) $9T27\simeq \PSL_2(\bF_8)$ and 
(3)-(ii) $10T11\simeq A_5\times C_2$ in Theorem \ref{thmain2}, 
we do not know whether $T$ is stably $k$-rational. \\
(2) For the reader's convenience, we note that:\\
(i) $8T2\simeq C_4\times C_2$, 
$8T3\simeq C_2\times C_2\times C_2$, 
$8T4\simeq D_4$, 
$8T5\simeq Q_8$, 
$8T6\simeq D_8$, 
$8T7\simeq M_{16}$,
$8T8\simeq QD_8$, 
$8T12\simeq \SL_2(\bF_3)$, 
$8T14\simeq S_4$, 
$8T23\simeq \GL_2(\bF_3)$, 
$8T37\simeq \PSL_2(\bF_7)\simeq \PSL_3(\bF_2)$, 
$8T49\simeq A_8$, 
$8T50\simeq S_8$.\\ 
(ii) $9T2\simeq C_3\times C_3$, 
$9T27\simeq \PSL_2(\bF_8)$, 
$9T33\simeq A_9$, 
$9T34\simeq S_9$.\\ 
(iii) $10T7\simeq A_5$,
$10T13\simeq S_5$, 
$10T26\simeq \PSL_2(\bF_9)\simeq A_6$, 
$10T30\simeq \PGL_2(\bF_9)$, 
$10T31\simeq M_{10}$, 
$10T32\simeq S_6$, 
$10T35\simeq {\rm P\Gamma L}_2(\bF_9)$, 
$10T44\simeq A_{10}$, 
$10T45\simeq S_{10}$.
\end{remark}

We organize this paper as follows. 
In Section \ref{sePre}, we prepare some basic tools to prove 
stably and retract rationality of algebraic tori. 
We also give known results about rationality problem for algebraic tori, 
in particular, norm one tori.
In Section \ref{seProof}, we will give the proof of 
Theorem \ref{thmain1} and Theorem \ref{thmain2} 
which are main theorems of this paper. 

\begin{acknowledgments}
The authors would like to thank 
Ming-chang Kang and Shizuo Endo 
for giving them useful and valuable comments. 
\end{acknowledgments}

%
\section{Preliminaries: rationality problem for algebraic tori and flabby resolution}\label{sePre}

We recall some basic facts of the theory of flabby (flasque) $G$-lattices
(see \cite{CTS77}, \cite{Swa83}, \cite[Chapter 2]{Vos98}, \cite[Chapter 2]{Lor05}, \cite{Swa10}).

\begin{definition}
Let $G$ be a finite group and $M$ be a $G$-lattice 
(i.e. finitely generated $\bZ[G]$-module which is $\bZ$-free 
as an abelian group). \\
{\rm (i)} $M$ is called a {\it permutation} $G$-lattice 
if $M$ has a $\bZ$-basis permuted by $G$, 
i.e. $M\simeq \oplus_{1\leq i\leq m}\bZ[G/H_i]$ 
for some subgroups $H_1,\ldots,H_m$ of $G$.\\
{\rm (ii)} $M$ is called a {\it stably permutation} $G$-lattice 
if $M\oplus P\simeq P^\prime$ 
for some permutation $G$-lattices $P$ and $P^\prime$.\\
{\rm (iii)} $M$ is called {\it invertible} (or {\it permutation projective}) 
if it is a direct summand of a permutation $G$-lattice, 
i.e. $P\simeq M\oplus M^\prime$ for some permutation $G$-lattice 
$P$ and a $G$-lattice $M^\prime$.\\
{\rm (iv)} $M$ is called {\it flabby} (or {\it flasque}) if $\widehat H^{-1}(H,M)=0$ 
for any subgroup $H$ of $G$ where $\widehat H$ is the Tate cohomology.\\
{\rm (v)} $M$ is called {\it coflabby} (or {\it coflasque}) if $H^1(H,M)=0$
for any subgroup $H$ of $G$.
\end{definition}

\begin{lemma}[Lenstra {\cite[Propositions 1.1 and 1.2]{Len74}, see also Swan 
\cite[Section 8]{Swa83}}]\label{lemSL}
Let $E$ be an invertible $G$-lattice.\\
{\rm (i)} $E$ is flabby and coflabby.\\
{\rm (ii)} If $C$ is a coflabby $G$-lattice, then any short exact sequence
$0 \rightarrow C \rightarrow N \rightarrow E \rightarrow 0$ splits.
\end{lemma}

\begin{definition}[{see \cite[Section 1]{EM75}, \cite[Section 4.7]{Vos98}}]
Let $\cC(G)$ be the category of all $G$-lattices. 
Let $\cS(G)$ be the full subcategory of $\cC(G)$ of all permutation $G$-lattices 
and $\cD(G)$ be the full subcategory of $\cC(G)$ of all invertible $G$-lattices.
Let 
\begin{align*}
\cH^i(G)=\{M\in \cC(G)\mid \widehat H^i(H,M)=0\ {\rm for\ any}\ H\leq G\}\ (i=\pm 1)
\end{align*}
be the class of ``$\widehat H^i$-vanish'' $G$-lattices 
where $\widehat H^i$ is the Tate cohomology. 
Then we have the inclusions 
$\cS(G)\subset \cD(G)\subset \cH^i(G)\subset \cC(G)$ $(i=\pm 1)$. 
\end{definition}

\begin{definition}\label{defCM}
We say that two $G$-lattices $M_1$ and $M_2$ are {\it similar} 
if there exist permutation $G$-lattices $P_1$ and $P_2$ such that 
$M_1\oplus P_1\simeq M_2\oplus P_2$. 
We denote the similarity class of $M$ by $[M]$. 
The set of similarity classes $\cC(G)/\cS(G)$ becomes a 
commutative monoid 
(with respect to the sum $[M_1]+[M_2]:=[M_1\oplus M_2]$ 
and the zero $0=[P]$ where $P\in \cS(G)$). 
\end{definition}

\begin{theorem}[Endo and Miyata {\cite[Theorem 3.3]{EM75}, Endo and Kang {\cite[Theorem 1.4]{EK17}}}]
Let $G$ be a finite group. 
Then the following conditions are equivalent:\\
{\rm (i)} The commutative monoid $\cH^{-1}(G)/\cS(G)$ is a finite group;\\
{\rm (ii)} $G=C_n$, $G=D_m$ $(m\geq 3: odd)$, 
$G=C_{q^f}\times D_m$ 
$(q: odd\ prime,\ f\geq 1,\ m\geq 3: odd,\ {\rm gcd}\{q,m\}=1)$ 
where $(\bZ/q^f\bZ)^\times=\langle \overline{p}\rangle$
for any prime divisor $p$ of $m$, or 
$G=Q_{4m}$ $(m\geq 3: odd)$ where $p\equiv 3\pmod{4}$ for
any prime divisor $p$ of $m$. 
\end{theorem}
\begin{theorem}[{Endo and Miyata \cite[Lemma 1.1]{EM75}, Colliot-Th\'el\`ene and Sansuc \cite[Lemma 3]{CTS77}, 
see also \cite[Lemma 8.5]{Swa83}, \cite[Lemma 2.6.1]{Lor05}}]\label{thEM}
For any $G$-lattice $M$,
there exists a short exact sequence of $G$-lattices
$0 \rightarrow M \rightarrow P \rightarrow F \rightarrow 0$
where $P$ is permutation and $F$ is flabby.
\end{theorem}
\begin{definition}\label{defFlabby}
The exact sequence $0 \rightarrow M \rightarrow P \rightarrow F \rightarrow 0$ 
as in Theorem \ref{thEM} is called a {\it flabby resolution} of the $G$-lattice $M$.
$\rho_G(M)=[F] \in \cC(G)/\cS(G)$ is called {\it the flabby class} of $M$,
denoted by $[M]^{fl}=[F]$.
Note that $[M]^{fl}$ is well-defined: 
if $[M]=[M^\prime]$, $[M]^{fl}=[F]$ and $[M^\prime]^{fl}=[F^\prime]$
then $F \oplus P_1 \simeq F^\prime \oplus P_2$
for some permutation $G$-lattices $P_1$ and $P_2$,
and therefore $[F]=[F^\prime]$ (cf. \cite[Lemma 8.7]{Swa83}). 
We say that $[M]^{fl}$ is {\it invertible} if 
$[M]^{fl}=[E]$ for some invertible $G$-lattice $E$. 
\end{definition}

For $G$-lattice $M$, 
it is not difficult to see 
\begin{align*}
\textrm{permutation}\ \ 
\Rightarrow\ \ 
&\textrm{stably\ permutation}\ \ 
\Rightarrow\ \ 
\textrm{invertible}\ \ 
\Rightarrow\ \ 
\textrm{flabby\ and\ coflabby}\\
&\hspace*{8mm}\Downarrow\hspace*{34mm} \Downarrow\\
&\hspace*{7mm}[M]^{fl}=0\hspace*{10mm}\Rightarrow\hspace*{5mm}[M]^{fl}\ 
\textrm{is\ invertible}.
\end{align*}

The above implications in each step cannot be reversed 
(see, for example, \cite[Section 1]{HY17}). 

Let $L/k$ be a finite Galois extension with Galois group $G={\rm Gal}(L/k)$ 
and $M$ be a $G$-lattice. 
The flabby class $\rho_G(M)=[M]^{fl}$ 
plays crucial role in the rationality problem for 
$L(M)^G$ as follows (see Voskresenskii's fundamental book \cite[Section 4.6]{Vos98} and Kunyavskii \cite{Kun07}, see also e.g. Swan \cite{Swa83}, 
Kunyavskii \cite[Section 2]{Kun90}, 
Lemire, Popov and Reichstein \cite[Section 2]{LPR06}, 
Kang \cite{Kan12}, 
Yamasaki \cite{Yam12}):  
\begin{theorem}[Endo and Miyata, Voskresenskii, Saltman]\label{thEM73}
Let $L/k$ be a finite Galois extension with Galois group $G={\rm Gal}(L/k)$. 
Let $M$ and $M^\prime$ be $G$-lattices.\\
{\rm (i)} $(${\rm Endo and Miyata} \cite[Theorem 1.6]{EM73}$)$ 
$[M]^{fl}=0$ if and only if $L(M)^G$ is stably $k$-rational.\\
{\rm (ii)} $(${\rm Voskresenskii} \cite[Theorem 2]{Vos74}$)$ 
$[M]^{fl}=[M^\prime]^{fl}$ if and only if $L(M)^G$ and $L(M^\prime)^G$ 
are stably $k$-isomorphic.\\
{\rm (iii)} $(${\rm Saltman} \cite[Theorem 3.14]{Sal84}$)$ 
$[M]^{fl}$ is invertible if and only if $L(M)^G$ is 
retract $k$-rational.
\end{theorem}
\begin{theorem}[{Endo and Miyata \cite[Theorem 1.5]{EM75}, see also Theorem \ref{th13-1}}]
Let $G$ be a finite group. 
Then the following conditions are equivalent:\\
{\rm (i)} $[J_G]^{fl}$ is invertible;\\
{\rm (ii)} all the Sylow subgroups of $G$ are cyclic;\\
{\rm (iii)} $\cH^{-1}(G)=\cH^1(G)=\cD(G)$, i.e. 
any flabby $($resp. coflabby$)$ $G$-lattice is invertible;\\
{\rm (iv)} the commutative monoid $\cH^{-1}(G)/\cS(G)$ is a group.
\end{theorem}
\begin{theorem}[{Colliot-Th\'{e}l\`{e}ne and Sansuc \cite[Corollaire 1]{CTS77}}]
\label{thCTS77}
Let $G$ be a finite group. 
Then the following conditions are equivalent:\\
{\rm (i)} $[J_G]^{fl}$ is coflabby;\\
{\rm (ii)} any Sylow subgroup of $G$ is cyclic or generalized quaternion 
$Q_{4n}$ of order $4n$ $(n\geq 2)$;\\
{\rm (iii)} any abelian subgroup of $G$ is cyclic;\\
{\rm (iv)} $H^3(H,\bZ)=0$ for any subgroup $H$ of $G$.
\end{theorem}
\begin{remark} 
(1) It is known that 
each of the conditions (i)--(iv) of Theorem \ref{thCTS77} 
is equivalent to the condition that $G$ has periodic cohomology, 
i.e. there exist 
$q\neq 0$ and $u\in\widehat H^q(G,\bZ)$ such that the cup product map 
$u\cup -: \widehat H^n(G,\bZ)\rightarrow \widehat H^{n+q}(G,\bZ)$ 
is an isomorphism for any $n\in\bZ$ (see \cite[Theorem 11.6]{CE56}).\\
(2) $H^3(H,\bZ)\simeq H^1(H,[J_G]^{fl})$ for any subgroup $H$ of $G$ 
(see \cite[Theorem 7]{Vos70} and \cite[Proposition 1]{CTS77}). 
\end{remark}

\begin{theorem}[{Endo and Miyata \cite[Theorem 2.1]{EM82}}]\label{thEM82}
Let $G$ be a finite group. 
Then the following conditions are equivalent:\\
{\rm (i)} $\cH^1(G)\cap \cH^{-1}(G)=\cD(G)$, i.e. 
any flabby and coflabby $G$-lattice is invertible;\\
{\rm (ii)} $[J_G\otimes_\bZ J_G]^{fl}=[[J_G]^{fl}]^{fl}$ is invertible;\\
{\rm (iii)} any $p$-Sylow subgroup of $G$ is cyclic for odd $p$ and 
cyclic or dihedral $($including Klein's four group$)$ for $p=2$.
\end{theorem}
Note that $[J_G]^{fl}=[J_G\otimes_\bZ J_G]$ (see \cite[Section 2]{EM82}). 
\begin{lemma}[{Swan \cite[Lemma 3.1]{Swa10}}]\label{lemSwa}
Let $0\rightarrow M_1\rightarrow M_2\rightarrow M_3\rightarrow 0$ 
be a short exact sequence of $G$-lattices with $M_3$ invertible. 
Then the flabby class $[M_2]^{fl}=[M_1]^{fl}+[M_3]^{fl}$. 
In particular, if $[M_1]^{fl}$ is invertible, 
then $-[M_1]^{fl}=[[M_1]^{fl}]^{fl}$. 
\end{lemma}
\begin{definition}\label{defMG} 
Let $G$ be a finite subgroup of $\GL_n(\bZ)$. 
{\it The $G$-lattice $M_G$ of rank $n$} 
is defined to be the $G$-lattice with a $\bZ$-basis $\{u_1,\ldots,u_n\}$ 
on which $G$ acts by $\sigma(u_i)=\sum_{j=1}^n a_{i,j}u_j$ for any $
\sigma=[a_{i,j}]\in G$. 
\end{definition}
\begin{lemma}[{see \cite[Remarque R2, page 180]{CTS77}, \cite[Lemma 2.17]{HY17}}]\label{lemp3}
Let $G$ be a finite subgroup of $\GL_n(\bZ)$ 
and $M_G$ be the corresponding $G$-lattice 
as in Definition \ref{defMG}. 
Let $H\leq G$ and $\rho_H(M_H)$ be the flabby class of $M_H$ 
as an $H$-lattice.\\
{\rm (i)} If $\rho_G(M_G)=0$, then $\rho_H(M_H)=0$.\\
{\rm (ii)} If $\rho_G(M_G)$ is invertible, then $\rho_H(M_H)$ is invertible.
\end{lemma}

\medskip
{\it Rationality problem for algebraic tori of small dimension.} 
It is easy to see that all the $1$-dimensional algebraic $k$-tori $T$, 
i.e. the trivial torus $\bG_m$ and the norm one torus $R_{K/k}^{(1)}(\bG_m)$ 
with $[K:k]=2$, are $k$-rational. 

There are $13$ (resp. $73$, $710$, $6079$) $\bZ$-classes forming 
$10$ (resp. $32$, $227$, $955$) $\bQ$-classes 
in $\GL_2(\bZ)$ (resp. $\GL_3(\bZ)$, $\GL_4(\bZ)$, $\GL_5(\bZ)$).
\begin{theorem}[{Voskresenskii \cite{Vos67}}]\label{thVo}
All the $2$-dimensional algebraic $k$-tori $T$ are $k$-rational. 
In particular, for any finite subgroups $G\leq \GL_2(\bZ)$, $L(x_1,x_2)^G$ 
is $k$-rational. 
\end{theorem}

A rational (stably rational, retract rational) classification of 
$3$-dimensional $k$-tori is given by Kunyavskii \cite{Kun90} 
(for the last statement, see Kang \cite[page 25, the fifth paragraph]{Kan12}). 
\begin{theorem}[{Kunyavskii \cite{Kun90}}]\label{thKu}
Let $L/k$ be a Galois extension and $G\simeq 
{\rm Gal}(L/k)$ be a finite subgroup of $\GL_3(\bZ)$ 
which acts on $L(x_1,x_2,x_3)$ via $(\ref{acts})$. 
Then $L(x_1,x_2,x_3)^G$ is not $k$-rational if and only if 
$G$ is conjugate to one of the $15$ groups which are given 
as in \cite[Theorem $1$]{Kun90}. 
Moreover, if $L(x_1,x_2,x_3)^G$ is not $k$-rational, 
then it is not retract $k$-rational. 
\end{theorem}
Denote $L(M)=L(x_1,\ldots,x_n)$ where 
$M$ is a $G$-lattice via the action $(\ref{acts})$. 
When $M$ is decomposable, we have the following 
by Theorem \ref{thEM73} and Theorem \ref{thKu} 
(Note that $[M_1\oplus M_2]^{fl}=[M_1]^{fl}+[M_2]^{fl}$.).
\begin{theorem}[{see also Hoshi, Kang and Kitayama \cite[Theorem 6.5]{HKK14}}]\label{thd}
Let $G\simeq {\rm Gal}(L/k)$ be a finite group and 
$M\simeq M_1\oplus M_2$ be a decomposable $G$-lattice.\\
{\rm (i)} $L(M)^G$ is retract $k$-rational 
if and only if both of $L(M_i)^G$ $(i=1,2)$ are retract $k$-rational.\\
{\rm (ii)} If $L(M_1)^G$ and $L(M_2)^G$ are stably $k$-rational 
$($resp. $k$-rational$)$, 
then $K(M)^G$ is stably $k$-rational $($resp. $k$-rational$)$.\\
{\rm (iii)} When \rank $M_i\leq 3$ $(i=1,2)$, 
$L(M)^G$ is stably $k$-rational $($resp. $k$-rational$)$ if and only if 
both of $L(M_i)^G$ $(i=1,2)$ are stably $k$-rational $($resp. $k$-rational$)$.
\end{theorem}
\begin{remark}
Theorem \ref{thd} (iii) does not hold in higher dimensions. 
Indeed, there exist two tori $T$ and $T^\prime$ of dimension $4$ 
which are not stably $k$-rational but the torus $T\times T^\prime$ 
of dimension $8$ is stably $k$-rational, 
i.e. $-[M_1]^{fl}=[M_2]^{fl}\neq 0$ 
(see \cite[Theorem 1.27 and Remark 1.29]{HY17}).
\end{remark}
A classification of stably/retract rational 
algebraic $k$-tori in dimensions $4$ and $5$ is given as follows: 
%
\begin{theorem}[Hoshi and Yamasaki {\cite[Theorem 1.9]{HY17}}]\label{th1}
Let $L/k$ be a Galois extension and $G\simeq 
{\rm Gal}(L/k)$ be a finite subgroup of $\GL_4(\bZ)$ 
which acts on $L(x_1,x_2,x_3,x_4)$ via $(\ref{acts})$. \\
{\rm (i)} 
$L(x_1,x_2,x_3,x_4)^G$ is stably $k$-rational 
if and only if 
$G$ is conjugate to one of the $487$ groups which are not in 
\cite[{\rm Tables} $2$, $3$ and $4$]{HY17}.\\
{\rm (ii)} 
$L(x_1,x_2,x_3,x_4)^G$ is not stably but retract $k$-rational 
if and only if $G$ is conjugate to one of the $7$ groups which are 
given as in \cite[{\rm Table} $2$]{HY17}.\\
{\rm (iii)} 
$L(x_1,x_2,x_3,x_4)^G$ is not retract $k$-rational if and only if 
$G$ is conjugate to one of the $216$ groups which are given as 
in \cite[{\rm Tables} $3$ and $4$]{HY17}.
\end{theorem}
\begin{theorem}[{Hoshi and Yamasaki \cite[Theorem 1.12]{HY17}}]\label{th2}
Let $L/k$ be a Galois extension and $G\simeq 
{\rm Gal}(L/k)$ be a finite subgroup of $\GL_5(\bZ)$ 
which acts on $L(x_1,x_2,x_3,x_4,x_5)$ via $(\ref{acts})$.\\
{\rm (i)} 
$L(x_1,x_2,x_3,x_4,x_5)^G$ is stably $k$-rational if and only if 
$G$ is conjugate to one of the $3051$ groups which are not in 
\cite[{\rm Tables} $11$, $12$, $13$, $14$ and $15$]{HY17}.\\
{\rm (ii)} 
$L(x_1,x_2,x_3,x_4,x_5)^G$ is not stably but retract $k$-rational 
if and only if $G$ is conjugate to one of the $25$ groups which are given as 
in \cite[{\rm Table} $11$]{HY17}.\\
{\rm (iii)} 
$L(x_1,x_2,x_3,x_4,x_5)^G$ is not retract $k$-rational if and only if 
$G$ is conjugate to one of the $3003$ groups which are given as 
in \cite[{\rm Tables} $12$, $13$, $14$ and $15$]{HY17}.
\end{theorem}

\section{Proof of Theorem {\ref{thmain1}} and Theorem \ref{thmain2}}
\label{seProof}

{\it Proof of Theorem \ref{thmain1}.} 
We may assume that 
$H$ is the stabilizer of one of the letters in $G$ 
(see the fourth paragraph of Section \ref{seInt}).

(1) and (2) follow from Theorem \ref{th13-2} and Theorem \ref{th15}. 

(3) follows from Theorem \ref{thS}. 
In particular, $T$ is retract $k$-rational for all the cases 
in this theorem because 
$[J_{G/H}]^{fl}$ is invertible for any transitive subgroup $G\leq S_n$ 
by Theorem \ref{thS} and Lemma \ref{lemp3} (ii). 

(4) follows from Theorem \ref{thA}. 

For (5) and (6), 
$C_{11}\rtimes C_5\leq \PSL_2(\bF_{11})$, 
$C_{11}\rtimes C_5\leq M_{11}$ 
and $C_{23}\rtimes C_{11}\leq M_{23}$ are transitive subgroups 
in $S_{11}$, $S_{11}$ and $S_{23}$ respectively. 
Hence we have $[J_{G/H}]^{fl}\neq 0$ for $G=\PSL_2(\bF_{11})$, $M_{11}$, 
$M_{23}$ by (2) and Lemma \ref{lemp3} (i). 

For (7), 
it is enough to show that $[J_{G/H}]^{fl}\neq 0$ for 
$G=\PSL_d(\bF_q)$ $(d\geq 3)$ by Lemma \ref{lemp3} (i). 
By the following lemma (Lemma \ref{lemP}), there exists 
a subgroup $C_p\rtimes C_d\leq G$ which is transitive in $S_p$. 
Hence we obtain that $[J_{G/H}]^{fl}\neq 0$ 
by Theorem \ref{th15} and Lemma \ref{lemp3} (i). 

For (8), we define $G_{m^\prime}:=\PSL_2(\bF_{2^e})\rtimes C_{2^{m^\prime}}$ 
where $0\leq m^\prime\leq m$ and $e=2^m$. 
Then $G_0=\PSL_2(\bF_{2^e})<G=G_{m^\prime}\leq G_m
={\rm P\Gamma L}_2(\bF_{2^e})$ with $1\leq m^\prime\leq m$. 
For $0\leq m^\prime\leq m$, define the normalizer 
$H_{m^\prime}=N_{G_{m^\prime}}({\rm Syl}_p(G_{m^\prime}))$ of 
${\rm Syl}_p(G_{m^\prime})\simeq C_p$ in $G_{m^\prime}$ 
where ${\rm Syl}_p(G_{m^\prime})$ is a $p$-Sylow subgroup of $G_{m^\prime}$. 
By Lemma \ref{lemP}, we obtain that $H_0\simeq D_p$ 
and $[G_0:H_0]=|G_0|/|H_0|=2^e(2^e-1)(2^e+1)/(2(2^e+1))
=2^{e-1}(2^e-1)$. 
We see that $[G_{m^\prime}:H_{m^\prime}]$ 
equals $[G_{m^\prime-1}:H_{m^\prime-1}]$ or 
$2\times [G_{m^\prime-1}:H_{m^\prime-1}]$. 
But the latter is impossible because 
$[G_{m^\prime}:H_{m^\prime}]\equiv 1\pmod{p}$ for any 
$0\leq m^\prime\leq m$ by Sylow theorem. 
Hence $[G_{m^\prime}:H_{m^\prime}]=[G_0:H_0]$ 
for $1\leq m^\prime\leq m$. 
This implies that 
$H_{m^\prime}\simeq C_p\rtimes C_{2^{m^\prime+1}}\leq G=G_{m^\prime}$ 
$(1\leq m^\prime\leq m)$. 
Because $H_{m^\prime}\leq S_p$ is transitive, 
it follows from Theorem \ref{th15} and Lemma \ref{lemp3} (i) 
that $[J_{G/H}]^{fl}\neq 0$. 
\qed

\begin{lemma}\label{lemP}
Let $d\geq 2$ be an integer, $q$ be a prime power and 
$p=(q^d-1)/(q-1)$ be a prime number. 
Let $G=\PSL_d(\bF_q)$ be a transitive subgroup of $S_p$, 
${\rm Syl}_p(G)\simeq C_p$ be a $p$-Sylow subgroup of $G$ 
and $H=N_G({\rm Syl}_p(G))$ be the normalizer of ${\rm Syl}_p(G)$ in $G$. 
Then $d\mid p-1$ is a prime number and 
$H\simeq C_p\rtimes C_d$ is a transitive subgroup of $S_p$ of order $pd$. 
\end{lemma}
\begin{proof}
Step 1: $d$ is a prime number. 
Suppose not. Then $d=ab$ and $p=(q^{ab}-1)/(q-1)=q^{ab-1}+\cdots+q+1=(q^{(a-1)b}+\cdots+q^b+1)(q^{b-1}+\cdots+q+1)$. Contradiction. 

Step 2: $p>d$. Because $q\geq 2$, we have $p=q^{d-1}+\cdots+q+1>d$. 

Step 3: $\gcd\{d,q-1\}=1$. Suppose not. 
Then we have $d\mid q-1$ since $d$ is a prime number by Step 1. 
This implies that $q\equiv 1\pmod{d}$ and 
hence $p=q^{d-1}+\cdots+q+1\equiv 0\pmod{d}$. 
But it is impossible because $p>d$ by Step 2 and 
both $p$ and $d$ are prime numbers by Step 1. 

Step 4: $|G|=p\cdot \prod_{i=1}^{d-1}(q^d-q^i)$. 
By the definition of $G$, we see that 
$|G|=\frac{|\GL_d(\bF_q)|}{(q-1)\gcd\{d,q-1\}}$. 
Hence we have $|G|=(\prod_{i=0}^{d-1}(q^d-q^i))/(q-1)
=p\cdot \prod_{i=1}^{d-1}(q^d-q^i)$ by 
Step 3. 

Step 5: $\frac{|G|}{p}\equiv d\pmod{p}$. 
Define $f(X)=\prod_{i=1}^{d-1}(X^d-X^i)$ and $\Phi_d(X)=X^{d-1}+\cdots+X+1$. 
Let $\zeta_d$ be a primitive $d$-th root of unity. 
For $a=1,\ldots,d-1$, we have 
$f(\zeta_d^a)=\prod_{i=1}^{d-1}(1-\zeta_d^{ai})=\Phi_d(1)=d$. 
This implies that $f(X)\equiv d\pmod{\Phi_d(X)}$. 
By Step 4, specializing $X$ to $q$, we have 
$\frac{|G|}{p}=f(q)\equiv d\pmod{p}$ because $\Phi_d(q)=p$. 

Step 6: $|H|\equiv pd\pmod{p^2}$. 
Let $s$ be the number of $p$-Sylow subgroups of $G$. 
Note that ${\rm Syl}_p(G)\simeq C_p$ because $G$ is transitive in $S_p$. 
By Sylow theorem and the definition of $H=N_G({\rm Syl}_p(G))$, 
we have $s\equiv 1\pmod{p}$ and $s=[G:H]=|G|/|H|$. 
Hence $\frac{|H|}{p}\equiv d\pmod{p}$ by Step 5. 
This implies that $|H|\equiv pd\pmod{p^2}$. 

Step 7: $H\simeq C_p\rtimes C_d$ is transitive in $S_p$. 
From the definition of $H$, we see that 
$H\simeq C_p\rtimes C_b$ for some $b\mid p-1$ and $H$ is transitive in $S_p$. 
We know that $2\leq d<p$ by Step 2. 
Hence, by Step 6, $|H|=pd$ and $b=d$. 
\end{proof}
%
{\it Proof of Theorem \ref{thmain2}.}
We may assume that 
$H$ is the stabilizer of one of the letters in $G$ 
(see the fourth paragraph of Section \ref{seInt}).

(1) The case $8Tm$ $(1\leq m\leq 50)$. 

(1-1) The case where $K/k$ is Galois: $1\leq m\leq 5$. 
For $8T1\simeq C_8$, $8T2\simeq C_4\times C_2$, 
$8T3\simeq C_2\times C_2\times C_2$, 
$8T4\simeq D_4$ and $8T5\simeq Q_8$, 
$K/k$ is a Galois extension.
Hence, it follows from Theorem \ref{th13-1} that 
$T$ is not retract $k$-rational for $8T2$, $8T3$, $8T4$ and $8T5$. 
By Theorem \ref{th13-2}, $T$ is stably $k$-rational for $8T1$. 

(1-2) The case where $K/k$ is not Galois: $6\leq m\leq 50$. 
Let $L/k$ be a Galois closure of $K/k$. 
If $G={\rm Gal}(L/k)$ is a $2$-group, 
then by Theorem \ref{th14} the flabby class $\rho_G(J_{G/H})=[J_{G/H}]^{fl}$ 
is not invertible and $T$ is not retract $k$-rational. 
Hence we assume that $G={\rm Gal}(L/k)$ is not a $2$-group. 
Take a $2$-Sylow subgroup $G_2$ of $G$. 
Then we see that $G_2$ is transitive in $S_8$ 
and is not cyclic. 
Then, again by Theorem \ref{th14}, we get 
the flabby class $\rho_{G_2}(J_{G_2/H_2})$ is not invertible 
where $H_2$ is a $2$-Sylow subgroup of $H$. 
Hence it follows from Lemma \ref{lemp3} (ii) 
that $\rho_G(J_{G/H})$ is not invertible 
and $T$ is not retract $k$-rational.

(2) The case $9Tm$ $(1\leq m\leq 34)$. 

(2-1) The case where $K/k$ is Galois: $m=1,2$. 
For $9T1\simeq C_9$ and $9T2\simeq C_3\times C_3$, 
$K/k$ is a Galois extension. 
Then it follows from Theorem \ref{th13-1} and Theorem \ref{th13-2} 
that $T$ is stably $k$-rational for $9T1$ and 
$T$ is not retract $k$-rational for $9T2$. 

(2-2) The case where $K/k$ is not Galois: $3\leq m\leq 34$. 
If $G={\rm Gal}(L/k)$ is a $3$-group 
where $L/k$ be a Galois closure of $K/k$, 
then $T$ is not retract $k$-rational by Theorem \ref{th14}. 
Thus we assume that $G$ is not a $3$-group. 
Take a $3$-Sylow subgroup $G_3$ of $G$. 
Then we see that $G_3$ is transitive in $S_9$ 
and is not cyclic except for $9T2\simeq D_9$ and $9T27\simeq \PSL_2(\bF_8)$. 
Hence $T$ is not retract $k$-rational as the same to (1) 
except for $9T3$ and $9T27$. 
For $9T3\simeq D_9$, $T$ is stably $k$-rational by Theorem \ref{th15}. 
For $9T27$, using the command  
{\tt IsInvertibleF(Norm1TorusJ(9,27))}, 
we see that $T$ is retract $k$-rational (see Example \ref{ex10} below). 
(We do not know whether $T$ is stably $k$-rational for $9T27\simeq \PSL_2(\bF_8)$.)

(3) The case $10Tm$ $(1\leq m\leq 45)$. 

(3-1) The case where $K/k$ is Galois: $m=1,2$. 
For $10T1\simeq C_{10}$ and $10T2\simeq D_5$, 
$K/k$ is a Galois extension. 
It follows from Theorem \ref{th13-2} that 
$T$ is stably $k$-rational for $10T1$ and $10T2$. 

(3-2) The case where $K/k$ is not Galois: $3\leq m\leq 45$. 

Case 1: $m=3$. 
For $10T3\simeq D_{10}$, 
using the command {\tt IsInvertibleF(Norm1TorusJ(10,3))}, 
we see that $[J_{G/H}]^{fl}$ is invertible and $T$ is retract $k$-rational. 
Using the method (Method III) given as in 
\cite[Section 5.7]{HY17}, we may take the flabby class 
$F=[J_{G/H}]^{fl}$ of rank $13$ and construct the isomorphism: 
$\bZ[G/C_2]\oplus\bZ[G/C_5]\oplus\bZ\simeq \bZ[G/D_5]\oplus F$ 
(see Example \ref{ex10}). 
Hence we conclude that $F=[J_{G/H}]^{fl}=0$ 
and $T$ is stably $k$-rational (see also \cite[Example 5.8]{HY17}). 

Case 2: $m=4, 5, 11, 12, 22$. 
For $m=22$, by using the function  
{\tt IsInvertibleF(Norm1TorusJ(10,22))}, 
we see that $[J_{G/H}]^{fl}$ is invertible and
$T$ is retract $k$-rational (see Example \ref{ex10}). 
For $m=4, 5, 11, 12$, 
$T$ is also retract $k$-rational because 
$10T4$, $10T5$, $10T11$, $10T12\leq 10T22$ and Lemma \ref{lemp3} (ii). 

For $m=4, 5$, using the function 
{\tt PossibilityOfStablyPermutationF(Norm1TorusJ(10,}$m${\tt ))}, 
we get that $[J_{G/H}]^{fl}\neq 0$ and 
$T$ is not stably $k$-rational.  
For $m=12, 22$, it follows from Lemma \ref{lemp3} (ii) 
and $10T4\simeq F_{20}\leq 10T12\simeq S_5$, 
$10T5\simeq F_{20}\times C_2\leq 10T22\simeq S_5\times C_2$, 
$T$ is also not stably $k$-rational 
(see Example \ref{ex10}). 
(We do not know whether $T$ is stably $k$-rational 
for $10T11\simeq A_5\times C_2$.) 

Case 3: $6\leq m\leq 35$ and $m\neq 11,12,22$. 
By the following inclusions of groups $G=10Tm$ and Lemma \ref{lemp3} (ii), 
it is enough to show that $T$ is not retract $k$-rational 
($[J_{G/H}]^{fl}$ is not invertible) for $m=6, 7, 8, 10, 18$:\\
$10T6\leq 10T9\leq 10T17, 10T19$,\\
$10T7\leq 10T13$,\\
$10T7\leq 10T26\leq 10T30, 10T31, 10T32\leq 10T35$,\\
$10T8\leq 10T14, {\bf 10T15}, 10T16$,\\
$10T10\leq 10T20, {\bf 10T21}\leq 10T27$,\\
$10T18\leq 10T28\leq 10T33, 10T42\leq 10T43, 10T44\leq 10T45$\\
and\\
${\bf 10T15}\leq 10T23, 10T24, 10T25\leq 10T29$,\\
${\bf 10T15}\leq 10T34\leq 10T36, 10T37, 10T38\leq 10T39$,\\
${\bf 10T21}\leq 10T40\leq 10T41$. 

For $m=6, 7, 8, 10, 18$, 
by using the function {\tt IsInvertibleF(Norm1TorusJ(10,}$m${\tt ))}, 
we may confirm that $[J_{G/H}]^{fl}$ is not invertible and 
$T$ is not retract $k$-rational.\qed\\

We give the GAP \cite{GAP} computations of the functions in the proof 
in Example \ref{ex10} below. 
Some related programs are also available from 
{\tt https://www.math.kyoto-u.ac.jp/\~{}yamasaki/Algorithm/RatProbAlgTori/}.

\begin{example}[Computations for $nTm\leq S_n$ with $n=8, 9, 10$]\label{ex10}
We give the demonstration of the GAP computations 
in the proof of Theorem \ref{thmain2} 
(see \cite[Chapter 5]{HY17} for the explanation of the functions). 



\begin{verbatim}
gap> Read("FlabbyResolution.gap");

gap> NrTransitiveGroups(8);
50
gap> t8:=List([1..50],x->TransitiveGroup(8,x));;
gap> Sy2t8:=List(t8,x->SylowSubgroup(x,2));;
gap> List(Sy2t8,x->Length(Orbits(x,[1..8])));
[ 1, 1, 1, 1, 1, 1, 1, 1, 1, 1, 1, 1, 1, 1, 1, 1, 1, 1, 1, 1, 1, 1, 1, 1, 1, 
  1, 1, 1, 1, 1, 1, 1, 1, 1, 1, 1, 1, 1, 1, 1, 1, 1, 1, 1, 1, 1, 1, 1, 1, 1 ]
gap> Filtered([1..50],x->IsCyclic(Sy2t8[x]));
[ 1 ]

gap> NrTransitiveGroups(9);
34
gap> t9:=List([1..34],x->TransitiveGroup(9,x));;
gap> Sy3t9:=List(t9,x->SylowSubgroup(x,3));;
gap> List(Sy3t9,x->Length(Orbits(x,[1..9])));
[ 1, 1, 1, 1, 1, 1, 1, 1, 1, 1, 1, 1, 1, 1, 1, 1, 1, 1, 1, 1, 1, 1, 1, 1, 1, 
  1, 1, 1, 1, 1, 1, 1, 1, 1 ]
gap> Filtered([1..34],x->IsCyclic(Sy3t9[x]));
[ 1, 3, 27 ]
gap> IsInvertibleF(Norm1TorusJ(9,27)); # T is retract k-rational for 9T27=PSL(2,8)
true

gap> NrTransitiveGroups(10);
45                          
gap> t10:=List([1..45],x->TransitiveGroup(10,x));
[ C(10)=5[x]2, D(10)=5:2, D_10(10)=[D(5)]2, 1/2[F(5)]2, F(5)[x]2, [5^2]2, 
  A_5(10), [2^4]5, [1/2.D(5)^2]2, 1/2[D(5)^2]2, A(5)[x]2, 1/2[S(5)]2=S_5(10a),
  S_5(10d), [2^5]5, [2^4]D(5), 1/2[2^5]D(5), [5^2:4]2, [5^2:4]2_2,            
  [5^2:4_2]2, [5^2:4_2]2_2, [D(5)^2]2, S(5)[x]2, [2^5]D(5), [2^4]F(5),        
  1/2[2^5]F(5), L(10)=PSL(2,9), [1/2.F(5)^2]2, 1/2[F(5)^2]2, [2^5]F(5),       
  L(10):2=PGL(2,9), M(10)=L(10)'2, S_6(10)=L(10):2, [F(5)^2]2, [2^4]A(5),     
  L(10).2^2=P|L(2,9), [2^5]A(5), [2^4]S(5), 1/2[2^5]S(5), [2^5]S(5),          
  [A(5)^2]2, [1/2.S(5)^2]2=[A(5):2]2, 1/2[S(5)^2]2, [S(5)^2]2, A10, S10 ]

gap> G:=Norm1TorusJ(10,3); # G=10T3=D10 
<matrix group with 2 generators>
gap> IsInvertibleF(G); # T is retract k-rational for 10T3
true
gap> mis:=SearchCoflabbyResolutionBase(TransposedMatrixGroup(G),2);;
gap> List(mis,Length); # searching suitable flabby class F
[ 30, 22, 30, 30, 22, 30 ]
gap> mi:=mis[2];;
gap> Rank(FlabbyResolutionFromBase(G,mi).actionF.1); # F is of rank 13 (=22-9)
13
gap> ll:=PossibilityOfStablyPermutationFFromBase(G,mi);
[ [ 1, -1, 0, -1, 2, 0, 0, 1, 1, -1, -1 ], 
  [ 0, 0, 1, 0, 0, 1, -1, 0, 0, 1, -1 ] ]
gap> l:=ll[Length(ll)];
[ 0, 0, 1, 0, 0, 1, -1, 0, 0, 1, -1 ]
gap> [l[3],l[6],l[10],l[7],l[11]];
[ 1, 1, 1, -1, -1 ]
gap> ss:=List(ConjugacyClassesSubgroups2(G),
> x->StructureDescription(Representative(x)));
[ "1", "C2", "C2", "C2", "C2 x C2", "C5", "D10", "D10", "C10", "D20" ]
gap> bp:=StablyPermutationFCheckPFromBase(G,mi,Nlist(l),Plist(l));;
gap> Length(bp);
19
gap> Length(bp[1]); # rank of the both sides of the isomorphism is 15
15
gap> rs:=RandomSource(IsMersenneTwister);
<RandomSource in IsMersenneTwister>
gap> rr:=List([1..10000],x->List([1..19],y->Random(rs,[-1..1])));;
gap> Filtered(rr,x->Determinant(x*bp)^2=1);
[ [ 0, 1, 0, -1, 0, -1, 0, -1, -1, 0, -1, 1, 0, 1, 1, 0, 0, 1, 1 ] ]
gap> p:=last[1]*bp;
[ [ 0, 1, 1, 0, 1, 0, 1, 0, 1, 0, 0, -1, -1, 0, 0 ], 
  [ 1, 0, 0, 1, 0, 1, 0, 1, 0, 1, -1, 0, 0, -1, 0 ], 
  [ -1, 0, -1, -1, 0, -1, -1, -1, 0, -1, 1, 0, 1, 1, 0 ], 
  [ 0, 0, 0, 0, 0, 0, 0, 0, 0, 0, 1, 2, 1, 2, 1 ], 
  [ -1, -1, -1, 0, -1, -1, -1, 0, -1, 0, 1, 1, 1, 0, 0 ], 
  [ -1, -1, 0, -1, -1, -1, 0, -1, 0, -1, 1, 1, 0, 1, 0 ], 
  [ 0, -1, -1, -1, -1, 0, -1, 0, -1, -1, 1, 1, 1, 0, 0 ], 
  [ -1, 0, -1, -1, 0, -1, 0, -1, -1, -1, 1, 1, 0, 1, 0 ], 
  [ -1, -1, -1, 0, -1, 0, -1, -1, -1, 0, 1, 1, 1, 0, 0 ], 
  [ -1, -1, 0, -1, 0, -1, -1, -1, 0, -1, 1, 1, 0, 1, 0 ], 
  [ 0, -1, -1, 0, -1, -1, -1, 0, -1, -1, 1, 1, 1, 0, 0 ], 
  [ -1, 0, 0, -1, -1, -1, 0, -1, -1, -1, 1, 1, 0, 1, 0 ], 
  [ 0, -1, -1, -1, -1, 0, -1, -1, -1, 0, 1, 1, 1, 0, 0 ], 
  [ 0, -1, -1, -1, -1, 0, -1, -1, -1, 0, 0, 1, 1, 1, 0 ], 
  [ -1, 0, 0, -1, -1, -1, 0, -1, -1, -1, 1, 0, 1, 1, 0 ] ]
gap> Determinant(p); # T is stably k-rational for 10T3
1
gap> StablyPermutationFCheckMatFromBase(G,mi,Nlist(l),Plist(l),p); 
true

gap> IsInvertibleF(Norm1TorusJ(10,22)); # T is retract k-rational for 10T22 
true

gap> PossibilityOfStablyPermutationF(Norm1TorusJ(10,4)); # T is not stably k-rational for 10T4
[ [ 1, 1, 2, 1, -1, 0, -2 ] ]
gap> PossibilityOfStablyPermutationF(Norm1TorusJ(10,5)); # T is not stably k-rational for 10T5
[ [ 1, 0, 0, 1, 0, 1, 1, 1, 0, 0, 0, -1, 0, -1, -1, 2, -2 ],
  [ 0, 1, 0, 0, -1, -1, -1, 0, 2, 0, -1, 0, 1, 1, 1, -2, 0 ],
  [ 0, 0, 1, 2, -1, 2, 2, 2, -2, -1, 0, -2, 1, -2, -2, 4, -2 ] ]

gap> IsInvertibleF(Norm1TorusJ(10,6)); # T is not retract k-rational for 10T6
false
gap> IsInvertibleF(Norm1TorusJ(10,7)); # T is not retract k-rational for 10T7
false
gap> IsInvertibleF(Norm1TorusJ(10,8)); # T is not retract k-rational for 10T8
false
gap> IsInvertibleF(Norm1TorusJ(10,10)); # T is not retract k-rational for 10T10
false
gap> IsInvertibleF(Norm1TorusJ(10,18)); # T is not retract k-rational for 10T18 
false

gap> sub:=List([1..45],x->Filtered([1..x], # checking subgroups of 10Tm (may not all)
> y->IsSubgroup(TransitiveGroup(10,x),TransitiveGroup(10,y))));
[ [ 1 ], [ 2 ], [ 1, 2, 3 ], [ 4 ], [ 1, 2, 3, 4, 5 ], [ 1, 6 ], [ 7 ], [ 8 ],
  [ 1, 2, 3, 6, 9 ], [ 10 ], [ 1, 2, 3, 11 ], [ 4, 12 ], [ 4, 7, 13 ],
  [ 1, 8, 14 ], [ 8, 15 ], [ 2, 8, 16 ], [ 1, 2, 3, 4, 5, 6, 9, 17 ], [ 18 ],
  [ 1, 2, 3, 6, 9, 19 ], [ 4, 10, 20 ], [ 1, 2, 3, 6, 9, 10, 21 ],
  [ 1, 2, 3, 4, 5, 11, 12, 22 ], [ 1, 2, 3, 8, 14, 15, 16, 23 ], [ 8, 15, 24 ],
  [ 4, 8, 15, 25 ], [ 26 ], [ 1, 2, 3, 4, 5, 6, 9, 10, 17, 19, 20, 21, 27 ],
  [ 18, 28 ], [ 1, 2, 3, 4, 5, 8, 14, 15, 16, 23, 24, 25, 29 ], [ 26, 30 ],
  [ 26, 31 ], [ 26, 32 ],
  [ 1, 2, 3, 4, 5, 6, 9, 10, 17, 18, 19, 20, 21, 27, 28, 33 ], [ 8, 15, 34 ],
  [ 26, 30, 31, 32, 35 ], [ 1, 2, 3, 8, 11, 14, 15, 16, 23, 34, 36 ],
  [ 8, 15, 24, 34, 37 ], [ 4, 8, 12, 15, 25, 34, 38 ],
  [ 1, 2, 3, 4, 5, 8, 11, 12, 14, 15, 16, 22, 23, 24, 25, 29, 34, 36, 37, 38, 39 ],
  [ 1, 2, 3, 6, 9, 10, 11, 21, 40 ],
  [ 1, 2, 3, 4, 5, 6, 9, 10, 11, 12, 17, 19, 20, 21, 22, 27, 40, 41 ],
  [ 18, 28, 42 ],
  [ 1, 2, 3, 4, 5, 6, 9, 10, 11, 12, 17, 18, 19, 20, 21, 22, 27, 28, 33, 40, 41,
      42, 43 ], [ 7, 8, 15, 18, 24, 26, 28, 31, 34, 37, 42, 44 ],
  [ 1, 2, 3, 4, 5, 6, 7, 8, 9, 10, 11, 12, 13, 14, 15, 16, 17, 18, 19, 20, 21,
      22, 23, 24, 25, 26, 27, 28, 29, 30, 31, 32, 33, 34, 35, 36, 37, 38, 39, 40,
      41, 42, 43, 44, 45 ] ]

gap> t1026:=TransitiveGroup(10,26); # checking that 10T7=A5 is a subgroup of 10T26=A6
L(10)=PSL(2,9)
gap> st1026:=Filtered(List(ConjugacyClassesSubgroups(t1026),
> Representative),x->Length(Orbits(x,[1..10]))=1);                                                   
[ Group([ (1,4)(2,5)(6,9)(7,8), (1,8,10)(2,6,7)(3,9,4) ]), Group([ (1,4)(2,5)  
  (6,9)(7,8), (1,10,6)(2,9,3)(4,8,5) ]), Group([ (1,4)(2,5)(6,9)               
  (7,8), (1,3,8,9)(4,5,10,6) ]) ]                                              
gap> List(st1026,StructureDescription); # 10T7=A5 is a subgroup of 10T26=A6
[ "A5", "A5", "A6" ]
\end{verbatim}
\end{example}

\end{document}